\documentclass[12pt,a4paper]{article}
\usepackage{amsmath,amssymb, amsthm, amsfonts, framed}
\usepackage{graphicx}
\usepackage{color}
\usepackage{verbatim}
\usepackage{hyperref}

\newtheorem{theorem}{Theorem}[section]
\newtheorem{lemma}[theorem]{Lemma}

\newtheorem{conjecture}[theorem]{Conjecture}
\newtheorem{corollary}[theorem]{Corollary}
\theoremstyle{definition}

\newtheorem{example}[theorem]{Example}
\theoremstyle{remark}
\newtheorem{remark}[theorem]{Remark}
\newcommand{\C}{{\mathbb C}}
\newcommand{\R}{{\mathbb R}}
\newcommand{\N}{{\mathbb N}}

\newcommand{\rme}{{\rm e}}

\newcommand{\cA}{{\cal A}}

\newcommand{\sig}{\sigma}

\newcommand{\gam}{\gamma}

\newcommand{\lam}{\lambda}
\newcommand{\del}{\delta}

\newcommand{\Spec}{{\rm Spec}}

\newcommand{\norm}{\Vert}
\renewcommand{\Re}{{\rm Re}}
\renewcommand{\Im}{{\rm Im}}

\newcommand{\dist}{{\rm dist}}

\newcommand{\Del}{\Delta}

\newenvironment{choices}{ \left\{ \begin{array}{ll} }{\end{array}\right.}

\newcommand{\move}[1]{}
\newcommand{\spec}[1]{\operatorname{Spec}(#1)}
\newcommand{\Ap}{\cA}
\newcommand{\ir}{\mathrm{i}}
\newcommand{\er}{\mathrm{e}}

\title{Spectra of a class of non-self-adjoint matrices}
\author{E. Brian Davies \thanks{{\bf EBD: }Department of Mathematics, King's College London, Strand,
London, WC2R 2LS, UK; E.Brian.Davies@kcl.ac.uk;  \url{http://www.mth.kcl.ac.uk/\~davies/}
}
\and
Michael Levitin\thanks{{\bf ML: }Department of Mathematics and Statistics, University of Reading, Whiteknights, PO Box 220, Reading RG6 6AX,
UK; m.levitin@reading.ac.uk; \url{http://www.personal.reading.ac.uk/\~ny901965/}}
}
\date{26 November 2013}
\renewcommand\footnotemark{}
\begin{document}
\maketitle
\begin{abstract}
We consider a new class of non-self-adjoint matrices that arise from an indefinite self-adjoint linear pencil of matrices, and obtain the spectral asymptotics of the spectra as the size of the matrices diverges to infinity. We prove that the spectrum is qualitatively different when a certain parameter $c$ equals $0$, and when it is non-zero, and that certain features of the spectrum depend on Diophantine properties of $c$.
\end{abstract}

\section{Introduction}\label{intro}

The spectral theory of pencils of linear operators has a long history, with contributions by distinguished people including Krein, Langer, Gohberg, Pontryagin and Shkalikov. It has many applications, for example to control theory, mathematical physics and vibrating structures. We refer to \cite{Markus} and \cite{TiMe} for accounts of this subject and extensive bibliographies. Among the theoretical tools that have been developed to study some such problems is the theory of Krein spaces, which also has a long history, \cite{La} and \cite{GLR}. There is also a substantial numerical literature on self-adjoint linear and quadratic pencils, \cite{PaCh90,Pa91,HTVD}. An interesting physically motivated example with some unusual features has recently been considered in \cite{ElLePo}.

It is well-known that self-adjoint pencils may have complex eigenvalues. If the pencil depends on a real parameter $c$ in addition to the spectral parameter, which we always call $\lambda$, one often sees two real eigenvalues of the pencil meeting at a square root singularity as $c$ changes, and then emerging as a complex conjugate pair, or vice versa. However, little has been written about the distribution of the complex eigenvalues, and the unexpected phenomena revealed in this paper show how hard a full understanding is likely to be. Some of these phenomena may disappear when studying suitable infinite-dimensional pencils of differential operators, but, if so, the reason for this will need to be explained.

The paper studies the simplest, non-trivial example of a finite-dimensional, self-adjoint, linear pencil, and in particular the asymptotic distribution of its non-real eigenvalues as the dimension increases. Our main results are presented in Theorem~\ref{th:c=0}, which is illustrated in Figure~\ref{fig:pict2}, and in Theorem~\ref{th:asymptcne0}, illustrated in Figure~\ref{fig:pict6}. However, a number of other cases exhibiting further complexities are also considered. Numerical studies indicate that some of the phenomena described here occur for a much larger class of pencils, including the case in which the matrix $D$ defined in \eqref{Ddef}
is indefinite and has slowly varying coefficients on each of the two subintervals concerned. Proving this is a task for the future.

A self-adjoint, linear pencil of $N\times N$ matrices is defined to be a family of matrices of the form $\Ap=\Ap(\lambda)=H-\lam D$, where $H,\, D$ are self-adjoint $ N\times N$ matrices and $\lam\in\C$. The number $\lam_0$ is said to be an \emph{eigenvalue} of this pencil $\Ap$ if $\Ap(\lam_0)=H-\lam_0 D$ is not invertible, or equivalently if $Hv=\lam_0 Dv$ has a non-zero solution $v\in \C^N$. The \emph{spectrum} of the pencil $\Ap$ is the set of all its eigenvalues,
and will be denoted by $\spec{\Ap}$. It coincides with the set of all roots of the polynomial
\[
p(\lam)=\det(H-\lam D).
\]
We always assume that $D$ is invertible, so that $p(\lam)$ is a polynomial of degree $N$ with non-zero leading coefficient. The spectrum of the pencil equals that of the matrix $D^{-1}H$, which is generically non-self-adjoint. In the standard case when $D$ is the identity matrix, the spectrum of the pencil $H-\lambda I$  coincides with the spectrum of the matrix $H$, which we also denote by $\spec{H}$.

The spectrum of such a self-adjoint pencil is real if either $H$ or $D$ is a definite matrix, i.e.\ all of its eigenvalues have the same sign. If both $H$ and $D$ are sign-indefinite matrices the problem is said to be indefinite, and it is known that the spectrum may then be complex. Numerical studies show that the eigenvalues of an indefinite self-adjoint pencil often lie on or under a small set of curves. Theorems~\ref{th:c=0} and \ref{th:asymptcne0} prove that this is true in two cases, and determine the curves asymptotically as $N\to\infty$. The analysis reduces to proving a similar statement for a certain class of large sparse polynomials. We also consider the algebraic multiplicities of the eigenvalues, and find that they may differ from the geometric multiplicities; see Theorem~\ref{n_is_m_plus_1_theorem}.

\section{Classes of problems and some general identities}

We consider the following class of problems. Fix an integer $N\in\N$, and define the $N\times N$ classes of matrices $H_{N;c}$ and $D_{m,n; \sigma, \tau}$, where
\[
H_{N;c}=\begin{pmatrix}
c & 1 & 0 & \dots & 0\\
1 & c & 1 & \dots & 0\\
 & \ddots & \ddots & \ddots & \\
 0 &  \dots & 1 & c & 1\\
 0 & \dots & 0 & 1 & c
\end{pmatrix}
\]
is tri-diagonal with the entries
\begin{equation}
(H_{N;c})_{r,s}=\begin{choices}
c &\quad\text{if }r=s,\\
1&\quad\text{if }|r-s|=1,\\
0&\quad\text{otherwise,}
\end{choices}\label{Hdef}
\end{equation}
where $c\in\R$ is a parameter, and
\begin{equation}\label{Ddef}
D_{m,n; \sigma, \tau}=\begin{pmatrix}
\sigma&&&&&\\
&\ddots&&&&\\
&&\sigma&&&\\
&&&\tau&&\\
&&&&\ddots&\\
&&&&&\tau
\end{pmatrix}\qquad
\begin{matrix}
\left.
\vphantom{\begin{matrix}\sigma&&\\
&\ddots&\\
&&\sigma
\end{matrix}
}
\right\}
\text{\small $m$ rows}
\\
\left.
\vphantom{\begin{matrix}\tau&&\\
&\ddots&\\
&&\tau
\end{matrix}
}
\right\}
\text{\small $n$ rows}
\end{matrix}
\end{equation}
is diagonal with the entries
\[
(D_{m,n; \sigma, \tau})_{r,s}=\begin{choices}
\sigma&\quad\text{if }r=s\le m,\\
\tau&\quad\text{if }m+1\le r=s\le m+n,\\
0&\quad\text{otherwise,}
\end{choices}
\]
where $m,n\in\N$ and $\sigma, \tau\in\C$ are parameters, and we assume $m+n=N$.

In the case $\sigma=-\tau=1$, we denote for brevity
\[
D_{m,n}:=D_{m,n;1,-1}=\begin{pmatrix}
1&&&&&\\
&\ddots&&&&\\
&&1&&&\\
&&&-1&&\\
&&&&\ddots&\\
&&&&&-1
\end{pmatrix}\qquad
\begin{matrix}
\left.
\vphantom{\begin{matrix}1&&\\
&\ddots&\\
&&1
\end{matrix}
}
\right\}
\text{\small $m$ rows}
\\
\left.
\vphantom{\begin{matrix}-1&&\\
&\ddots&\\
&&-1
\end{matrix}
}
\right\}
\text{\small $n$ rows}
\end{matrix}.
\]

We study the eigenvalues of the linear operator pencil
\[
\Ap_{m,n;c}=\Ap_{m,n;c}(\lambda)=H_{m+n;c}-\lambda D_{m,n}
\]
as $N=m+n\to\infty$. We shall write the eigenvalues as
\[
\lambda=u+\ir v/N,\qquad u,v\in\R.
\]
We justify the normalisation of the imaginary part of the eigenvalues later on.

We start with the following easy result on the localisation of eigenvalues of the pencil $\Ap_{m,n;c}$.

\begin{lemma}\label{th:location_rough}
\begin{enumerate}
\item[(a)] The spectrum $\spec{\Ap_{m,n;c}}$ is invariant under the symmetry $\lam\to\overline{\lam}$.
\item[(b)] All the eigenvalues $\lambda\in\spec{\Ap_{m,n;c}}$ satisfy
\[
|\lam|<2+|c|.
\]
\item[(c)] If $|c|\ge 2$, then $\spec{\Ap_{m,n;c}}\subset\R$.
\end{enumerate}
\end{lemma}

 In our asymptotic analysis, we concentrate on three special cases:
\begin{itemize}
 \item $m=n$, $c=0$;
 \item $m=n$, $c\ne 0$;
 \item $m\ne n$, $c=0$.
 \end{itemize}

 In some of these cases we can improve the localisation results of Lemma \ref{th:location_rough}, see Lemmas \ref{th:c0symm},  \ref{th:cne0symm},  and also cf. Conjecture \ref{conj:Relambdabound}.

 \begin{proof}[Proof of Lemma \ref{th:location_rough}]
 \begin{enumerate}
 \item[(a)]  This result is true for any pencil $H-\lambda D$ with self-adjoint coefficients $H$ and $D$: If $\lam$ is an eigenvalue of the pencil then $H-\lam D$ is not invertible, therefore $(H-\lam D)^\ast=H-\overline{\lam} D$ is not invertible and $\overline{\lam}$ is an eigenvalue.
 \item[(b)] A direct calculation shows that the eigenvalues of $H_{N;0}$, where $N=m+n$,  are given by
\begin{equation}\label{Heig}
\mu_j=2\cos(\pi j/(N+1)),\qquad 1\leq j\leq N.
\end{equation}
This establishes that $\norm H_{N;0}\norm=\mu_1=2\cos(\pi /(N+1))<2$. Thus $\norm H_{N;c}\norm=\norm H_{N;0}+cI\norm<2+|c|$. The eigenvalues of the pencil $\Ap_{m,n;c}$ coincide with the eigenvalues of the matrix
$D_{m,n}^{-1}H_{N;c}$. Therefore every eigenvalue satisfies $|\lam|\leq \norm D_{m,n}^{-1}H_{N;c}\norm = \norm H_{N;c}\norm <2+|c|$.
 \item[(c)]  If $c\geq 2$ then $H_{N;c}>0$ by \eqref{Heig}. Therefore
\[
\lam\in\spec{\Ap_{m,n;c}}\qquad\Longleftrightarrow\qquad\frac{1}{\lam}\in\spec{H_{N;c}^{-1/2}D_{m,n}H_{N;c}^{-1/2}}.
\]
Since $H_{N;c}^{-1/2}D_{m,n}H_{N;c}^{-1/2}$ is self-adjoint, the pencil $\Ap_{m,n;c}$ has real spectrum.

If $c\le-2$, then $H_{N; c}<0$ by  \eqref{Heig}, and the proof follows in a similar manner.
 \end{enumerate}
 \end{proof}

Our final results of this section reduces the eigenvalue problem for the pencil $\Ap_{m,n;c}$ to an explicit complex polynomial equation in two auxiliary variables.   We start by introducing some extra notation.

We shall always use the substitutions
\begin{equation}\label{zw_subs}
\lambda-c:=z+\frac{1}{z},\qquad \lambda+c:=w+\frac{1}{w},
\end{equation}
for the eigenvalues of the pencil $\cA_{m,n;c}$, where $z, w$ are, in general, some complex numbers. Note that each $\lambda$ corresponds to two values of $z$ (which are the solutions of the quadratic equation
\begin{equation}\label{eq:zeq}
z^2-(\lambda-c)z+1=0,
\end{equation}
and are inverses of each other), and two values of $w$ (which are the solutions of the quadratic equation
\begin{equation}\label{eq:weq}
w^2-(\lambda+c)w+1=0,
\end{equation}
which are also inverses of each other). If $c=0$, then $w=z$. If $\lam\notin\R$, we define $z,\, w$ to be the unique solutions of \eqref{eq:zeq},  \eqref{eq:weq}, resp., that satisfy
\begin{equation}\label{eq:z,w>1}
|z|>1,\qquad |w|> 1.
\end{equation}

We need also to introduce the families of meromorphic functions $\beta_{m,n}:\C^2\to\C$, $\gamma_{m,n}:\C^2\to\C$,  and $F_m:\C\to\C$ defined by
\begin{equation}
\beta_{m,n}(z,w)=(z^{m+1}-z^{-m-1})
(w^{n+1}-w^{-n-1})
+(z^m-z^{-m})
(w^n-w^{-n}),\label{betdef}
\end{equation}
\begin{equation}
\gam_{m,n}(z,w)=(z^{m+1}-z^{-m-1})
(w^{n+1}-w^{-n-1})
-(z^m-z^{-m})
(w^n-w^{-n}),\label{gamdef}
\end{equation}
and
\begin{equation}
\label{F_def}
F_m(z)=\frac{z^{m+1}-z^{-m-1}}{z^{m}-z^{-m}}.
\end{equation}
Obviously
\[
F_m(z)=\frac{\sinh((m+1)\log(z))}{\sinh({m}\log(z))}
\]
and the definition of $F_m$ can be thus also rewritten in terms of Chebyshev polynomials of the second kind. 
Ratios of orthogonal polynomials have been studied in considerable generality, see e.g. \cite{Nev}, \cite[Theorem 9.5.11]{Simon}, \cite{Simanek}, but the asymptotic properties that we use appear to be new, even for the simple case that we consider.

\begin{lemma}\label{lem:Fanotherbound}
Let $\tilde{F}_m:\C\cup\{\infty\}\to\C\cup\{\infty\}$, $m\ge 1$ denote the sequence of iteratively defined rational functions
\begin{equation}\label{eq:Ftildedef}
\tilde{F}_{m+1}(\zeta)=\zeta-\frac{1}{\tilde{F}_{m}(\zeta)},\qquad \tilde{F}_1(\zeta)=\zeta,\qquad\zeta\in\C\cup\{\infty\}.
\end{equation}
Then
\begin{enumerate}
\item[(a)] $\tilde{F}_m$ are Herglotz functions such that
\[
\Im(\tilde{F}_m(\zeta))\ge\Im(\zeta),\qquad\zeta\in\{z\in\C: \Im(z)>0\},
\]
with equality only for $m=1$.
\item[(b)] If
\begin{equation}\label{eq:z+1/zlowerbound}
\left|\zeta\right|>2,
\end{equation}
then
\begin{equation}\label{eq:Flowerbound}
|\tilde{F}_m(\zeta)|>1
\end{equation}
for all $m\in\N$.
\item[(c)]
\begin{equation}\label{eq:FandFtilde}
F_m(\zeta)=\tilde{F}_m\left(\zeta+\frac{1}{\zeta}\right).
\end{equation}
\end{enumerate}
\end{lemma}

\begin{proof}[Proof of Lemma \ref{lem:Fanotherbound}] To prove statements (a) and (b), we proceed by induction. First, for $m=1$, (a) and (b) are obvious. We also have
\[
F_1(\zeta)=\frac{\zeta^2-\zeta^{-2}}{\zeta-\zeta^{-1}}=\zeta+\frac{1}{\zeta},
\]
and so \eqref{eq:FandFtilde} for $m=1$ immediately follows from \eqref{eq:Ftildedef}.

Suppose now that (a) and (b) hold for some $m\in\N$. Then by \eqref{eq:Ftildedef}
\[
\Im(\tilde{F}_{m+1}(\zeta))=\Im\left(\zeta-\frac{1}{\tilde{F}_{m}(\zeta)}\right)=\Im(\zeta)+\frac{\Im(\tilde{F}_m(\zeta))}{|\tilde{F}_m(\zeta)|^2}>\Im(\zeta),
\]
and
\[
|\tilde{F}_{m+1}(\zeta))|=\left|\zeta-\frac{1}{\tilde{F}_{m}(\zeta)}\right|>|\zeta|-\frac{1}{|\tilde{F}_{m}(\zeta)|}>2-1=1,
\]
proving (a) and (b).
Additionally, by \eqref{F_def},
\[
\begin{split}
F_{m+1}(\zeta)&=\frac{\zeta^{m+2}-\zeta^{-(m+2)}}{\zeta^{m+1}-\zeta^{-(m+1)}}\\
&=\frac{\left(\zeta^{m+1}-\zeta^{-(m+1)}\right)\left(\zeta+\zeta^{-1}\right)-\zeta^{m}+\zeta^{-m}}{\zeta^{m+1}-\zeta^{-(m+1)}}\\
&=\zeta+\frac{1}{\zeta}-\frac{1}{F_{m}(\zeta)},
\end{split}
\]
and \eqref{eq:FandFtilde} follows from \eqref{eq:Ftildedef}, proving (c).
\end{proof}

Our first main result relates the eigenvalues of the pencil $\Ap_{m,n;c}$ with the functions $\beta_{m,n}$ and $F_m$.

\begin{theorem}\label{th:main1}
Let $\lambda\in\C\backslash \{-2-c,-2+c,2-c,2+c\}$ and let
$\lambda,z,w$ be related by \eqref{zw_subs}--\eqref{eq:z,w>1}. Then
\begin{enumerate}
\item[(a)] $\lambda$ is an eigenvalue of the pencil $\Ap_{m,n;c}$ if and only if
\begin{equation}\label{eq:beta0}
\beta_{m,n}(z,w)=0.
\end{equation}
\item[(b)] If  $\lambda$ is an eigenvalue of $\Ap_{m,n;c}$ then it is real if
and only if both $z$ and $w$ lie in the set
\[
(\R\backslash \{0\})\cup\{\zeta\in\C:|\zeta|=1\}.
\]
\item[(c)] If $\lambda\notin\R$ then $\lambda$ is an eigenvalue of
$\Ap_{m,n;c}$ if and only if
\begin{equation}
\label{F_condition}
F_m(z)F_n(w)=-1,
\end{equation}
where $\lam, z, w$ are related by  \eqref{zw_subs}--\,\eqref{eq:z,w>1}.
\end{enumerate}
\end{theorem}

The problem with using \eqref{F_condition} to characterize real
eigenvalues is the possibility that the numerator or denominator of
$F_m(z)$ or of $F_n(w)$ vanishes. Although one can treat all of the
special cases in turn, it is easier to revert to the use of part (a)
of the theorem.

The proof of Theorem \ref{th:main1}  is based on the following auxiliary result.

\begin{lemma}\label{sigtaulemma}
If $\sig=z+z^{-1}\not=\pm2$ and $\tau=w+w^{-1}\not=\pm2$ then
\begin{equation}
\det(H_{N;0}-D_{m,n;\sig,\tau})
=(-1)^{m+n}\frac{\gam_{m,n}(z,w)}
{(z-z^{-1})(w-w^{-1})}\label{detidentity}
\end{equation}
\end{lemma}

\begin{proof}[Proof of Lemma \ref{sigtaulemma}]
We prove that
\begin{equation}
\Del_{m,n}(z,w)=(z-z^{-1})(w-w^{-1})
\det(D_{m,n;\sig,\tau}-H_{N;0})\label{Deldef}
\end{equation}
satisfies $\Del=\gam$ by induction on $N=m+n$. If $N=1,\, 2$ this may be proved by direct computations for the five cases, in which $(m,n)$ equals one of $(1,0)$, $(0,1)$, $(2,0)$, $(1,1)$, $(0,2)$. Let $N\geq 3$ and suppose that (\ref{detidentity}) is known for all smaller values of $N$. The lower bound $N\geq 3$ implies that $m\geq 2$ or $n\geq 2$. Using the symmetry $(m,z)\leftrightarrow (n,w)$, we reduce to the case $m\geq 2$.

By expanding the determinant in (\ref{Deldef}) along the top row one obtains
\[
\Del_{m,n}(z,w)=(z+z^{-1})\Del_{m-1,n}(z,w)
-\Del_{m-2,n}(z,w).
\]
The inductive hypothesis yields
\begin{eqnarray*}
\Del_{m,n}(z,w)&=&(z+z^{-1})
(z^{m}-z^{-m})(w^{n+1}-w^{-n-1})\\
&&-(z+z^{-1})
(z^{m-1}-z^{1-m})(w^{n}-w^{-n})\\
&&-(z^{m-1}-z^{1-m})(w^{n+1}-w^{-n-1})\\
&&+(z^{m-2}-z^{2-m})(w^{n}-w^{-n})\\
&=& \gam_{m,n}(z,w),
\end{eqnarray*}
as required to prove the inductive step.
\end{proof}

\begin{proof}[Proof of Theorem \ref{th:main1}]
\begin{enumerate}
\item[(a)] Returning to the pencil $H_{N;c}-\lambda D_{n,m}$, we construct its characteristic polynomial from the obvious identity
\begin{equation}\label{pdef}
p_{n,m;c}(\lam):=\det(H_{N;c}-\lam D_{n,m})=\det(H_{N;0}-D_{n,m;\lambda-c, -c-\lambda}).
\end{equation}
Using \eqref{pdef}, \eqref{zw_subs}, \eqref{betdef}, \eqref{gamdef}, and Lemma \ref{sigtaulemma}, we obtain
\[
\begin{split}
p_{n,m;c}(\lam)&=(-1)^{m+n}\frac{\gam_{m,n}(z,-w)}{(z-z^{-1})((-w)-(-w)^{-1})}\\
&=(-1)^m\frac{\beta_{m,n}(z,w)}{(z-z^{-1})(w-w^{-1})}.
\end{split}
\]
We know already by Lemma \ref{th:location_rough} that any eigenvalue of $\Ap_{m,n;c}$ lies in the disk $\{\lam\in\C:|\lam|<2+|c|\}$.
As the conditions $\lam\ne \pm 2\pm c$ ensure that the denominator $(z-z^{-1})(w-w^{-1})$ does not vanish, the zeros of $p_{n,m;c}(\lam)$ are given by those of $\beta_{m,n}(z,w)$ and vice versa.
\item[(b)] If $\lambda$ is real, so are, by \eqref{zw_subs}, both $z+z^{-1}$ and $w+w^{-1}$. The result follows as $\Im\left(\zeta+\zeta^{-1}\right)= \Im(\zeta)\left(1-|\zeta|^{-2}\right)$ for $\zeta\in\C$.
\item[(c)] If $\lambda$ is non-real, then by part (b), $z^m\ne z^{-m}$ and $w^n\ne w^{-n}$. Dividing \eqref{eq:beta0} through by $(z^m-z^{-m})(w^n-w^{-n})$, and rearranging with account of \eqref{F_def}, we arrive
at \eqref{F_condition}.
\end{enumerate}
\end{proof}

\section{Asymptotic behaviour of non-real eigenvalues --- rough analysis}
The aim of this section is to show that the non-real eigenvalues of the pencil $\cA_{m,n;c}$ converge to  the real axis if \emph{both} $n,m\to\infty$. We also obtain a rough estimate for the rate of convergence, show that this estimate in principle cannot be improved, and show that the condition that both $n$ and $m$ go to infinity is necessary. Finally, we study the behaviour of some eigenvalues lying on the imaginary axis.

Recall that by Theorem \ref{th:main1}(c), $\lambda\in\spec{\cA_{m,n;c}}\setminus\R$ implies that
\begin{equation}\label{eq:modFprod=1}
|F_m(z) F_n(w)|=1.
\end{equation}

 Our results are built upon some estimates of the function $F_m$. Define a family of real monotone-increasing functions $G_m:(0,+\infty)\to\R$ by
 \[
 G_m(s):=\er^s\tanh(ms).
 \]

 \begin{lemma}\label{lem:F_est1} If $z=\er^{s+\ir\theta}$ where $s>0$, then
 \begin{equation}\label{eq:Fest1}
 |F_m(z)|> G_m(s)=G_m(\log(|z|)).
 \end{equation}
 \end{lemma}

 \begin{proof} We have
 \[
 \begin{split}
|F_m(z)|&= \left| \frac{z^{m+1}-z^{-m-1}}{z^m-z^{-m}} \right|
\geq \frac{|z|^{m+1}-|z|^{-m-1}}{|z|^m+|z|^{-m}}\\
&>\frac{\er^{(m+1)s}-\er^{-(m-1)s}}{\er^{ms} +\er^{-ms}}
=\frac{\er^s\sinh(ms)}{\cosh(ms)}=G_m(s).
\end{split}
\]
\end{proof}

\begin{corollary}\label{cor:FFlognlogm}
If $|z|\geq \exp(\log(m)/(2m))$ and
$|w|\geq \exp(\log(n)/(2n))$ then
\[
|F_m(z)F_n(w)|>1
\]
for all large enough $m,n$.
\end{corollary}

\begin{proof}
This follows directly from Lemma \ref{lem:F_est1} as
\[
\begin{split}
|F_m(z)|&>G_m(\log(|z|))\ge G_m\left(\frac{\log(m)}{2m}\right)\\
&= \exp\left(\frac{\log(m)}{2m}\right)
\frac{m^{1/2}-m^{-1/2}}{m^{1/2}+m^{-1/2}}\\
&=\left(1+\frac{\log(m)}{2m}(1+o(1))\right)
\left(1-\frac{2}{m}(1+o(1))\right)\\
&> 1
\end{split}
\]
for all large enough $m$, and from the same argument with $m$ replaced by $n$ and $z$ by $w$.
\end{proof}

We are now able to prove the main result of this section.

\begin{theorem}\label{th:crudebound}
The non-real eigenvalues of $\cA_{m,n;c}$ converge uniformly to the real axis as $n,m\to\infty$. More precisely,
\begin{equation}\label{eq:upperbounded}
\begin{split}
&\max\{ |\Im(\lam)|:\lam\in\spec{\cA_{m,n;c}}\}\\
&\qquad \leq\max\left\{\frac{\log(m)}{m}(1+o(1)), \frac{\log(n)}{n}(1+o(1))\right\}
\end{split}
\end{equation}
as $m,n\to\infty$.
\end{theorem}

\begin{proof} Suppose that $\lam\in\spec{\cA_{m,n;c}}\setminus\R$. Then, by Theorem \ref{th:main1}(c), $|F_m(z)F_n(w)|=1$, and by Corollary
\ref{cor:FFlognlogm} we have either
\begin{equation}\label{eq:logboundz}
|z|\geq \exp(\log(m)/(2m)),
\end{equation}
or
\begin{equation}\label{eq:logboundw}
|w|\geq \exp(\log(n)/(2n)).
\end{equation}

Suppose that \eqref{eq:logboundz} holds. Setting $z=\er^{s+\ir\theta}$ and using \eqref{eq:z,w>1} and Theorem \ref{th:main1}(b), we arrive at
\[
0<s<\frac{\log m}{2m}.
\]
Therefore, by  \eqref{zw_subs},
\begin{equation}\label{eq:logboundm}
|\Im(\lam)|=|\left(\rme^s-\rme^{-s}\right)
\sin(\theta)| \leq 2|\sinh(s)|\leq\frac{\log(m)}{m}(1+o(1))
\end{equation}
as $m\to\infty$.

If we assume \eqref{eq:logboundw} instead, we use $w$ instead of $z$ and arrive by the same argument at
\begin{equation}\label{eq:logboundn}
|\Im(\lam)|\leq\frac{\log(n)}{n}(1+o(1))
\end{equation}
as $n\to\infty$.

The result now follows by combining \eqref{eq:logboundm} and \eqref{eq:logboundn}.
\end{proof}

The next lemma provides a useful factorisation of $\beta_{m,m}(z,z)$, which we shall use on numerous occasions.

\begin{lemma}\label{lem:beta_factorisation} Define
\begin{align}
r_m^{(1)}(z)&=(z+\ir)z^{2m+1}-\ir(z-\ir),\label{r1def}\\
r_m^{(2)}(z)&= (z-\ir)z^{2m+1}+\ir(z+\ir).\label{r2def}
\end{align}
Then $\beta_{m,m}(z,z)=0$ if and only if either $r_m^{(2)}(z)=0$ or $r_m^{(2)}(\overline{z})=0$.
\end{lemma}

\begin{proof}
One may re-write \eqref{eq:beta0} in the form
\begin{equation}
\beta_{m,m}(z,z)=z^{-2m-2}r_m^{(1)}(z)r_m^{(2)}(z)=0.
\label{qtildemmz}
\end{equation}

The functions $r^{(j)}(z)$, $j=1,2$, defined by \eqref{r1def}, \eqref{r2def}, satisfy
\begin{equation}
\overline{r_m^{(2)}(z)}=r_m^{(1)}(\overline{z}),
\label{r1r2formula}
\end{equation}
so $r_m^{(1)}(\overline{z})=0$ if and only if $r_m^{(2)}(z)=0$. Moreover
\begin{equation}
r_m^{(j)}(z^{-1})=-z^{-2m-2} r_m^{(j)}(z),
\label{r1r2formula2}
\end{equation}
for $j=1,\, 2$, so $r_m^{(j)}(z^{-1})=0$ if and only if $r_m^{(j)}(z)=0$.
The formulae \eqref{qtildemmz}--\eqref{r1r2formula2} are checked by a direct calculation.
\end{proof}

The following lemma shows that the upper bound
\eqref{eq:upperbounded} is (in a sense) optimal. The decision to focus on purely imaginary $z$ in the theorem was based on numerical experiments, which indicate that such $z$ provide the greatest imaginary parts of the eigenvalues $\lam$.
We consider the case $m=n$ and $c=0$, so that $z=w$,  and show that the bound \eqref{eq:upperbounded} is attained by considering the odd values of $m$.

\begin{lemma}\label{z=iybound}
Suppose that $c=0$, $m=n$, $m$ is odd, and $z=w=\ir y$. The equation $r_m^{(2)}(\ir y)=0$ has  four solutions $y\in\R\setminus \{0\}$, symmetric with respect to zero, exactly one of which lies in $(1,\infty)$. That solution satisfies
\begin{equation}
y=1+\frac{\log(m)}{2m}(1+o(1))
\label{sharpybound}
\end{equation}
as $m\to\infty$. The corresponding eigenvalue $\lam\in \ir(0,\infty)$ of the pencil $\cA_{m,m;0}$ satisfies
\begin{equation}
\Im(\lam)=\frac{\log(m)}{m}(1+o(1))
\label{sharplambound}
\end{equation}
as $m\to\infty$.

If $m$ is even, then there are no solutions of $r_m^{(2)}(\ir y)=0$ for $y\in\R\setminus \{0\}$, and therefore no purely imaginary eigenvalues of $\cA_{m,m;0}$.
\end{lemma}

\begin{proof}
An elementary calculation shows that
\begin{equation}\label{eq:r2iy}
r_m^{(2)}(\ir y)= (-1)^{m+1}(y-1)y^{2m+1}-(y+1),\qquad\text{for all }m\in\N, y\in\R.
\end{equation}
Thus, if $m$ is odd,
\[
r_m^{(2)}(\ir y)= (y-1)y^{2m+1}-(y+1)
\]
Direct calculations show that $0,1,-1$ are not, in this case, solutions of $r_m^{(2)}(\ir y)=0$. If we put
\[
f_m(y)=y^{-1}r_m^{(2)}(\ir y)=(y-1)y^{2m}-1-\frac{1}{y}
\]
then $f_m$ is strictly monotonic increasing on $[1,\infty)$ with $f_m(1)<0$ and $f_m(y)\to +\infty$ as $y\to +\infty$. Therefore $f_m(y)=0$ has a unique solution in $(1,\infty)$, which we denote by $y_m$.

If $m\geq 3$ and $y=\rme^{\log(m)/2m}$ then $y>1$, $y^{2m}=m$ and
\begin{eqnarray*}
f_m(y)&>& (y-1) m-2\\
&>& \frac{\log(m)}{2m} m-2\\
&=& \frac{1}{2}(\log(m) -4)\\
&>&0
\end{eqnarray*}
provided $m\geq 55$. Therefore
\[
y_m<\exp\left(\frac{\log(m)}{2m}\right)
\]
for all such $m$.

Let
\[
y=\exp\left(\frac{\log(m)}{2m}(1-\del_m)\right)
\]
where $\del_m=(\log(m))^{-1/2}$ and $m\geq 3$. Then $y>1$ and $y^{2m}=\rme^{\log(m)(1-\del_m)}$. Therefore
\begin{eqnarray*}
f_m(y)&<& (y-1)\rme^{\log(m)(1-\del_m)}-1\\
&<& \frac{\log(m)}{2m}(1-\del_m)(1+o(1))\cdot m
\rme^{-(\log(m))^{-1/2}}-1\\
&<&  \frac{\log(m)}{2}(1-\del_m)(1+o(1))\cdot
\frac{3!}{(\log(m))^{3/2}}-1\\
&=& 3(\log(m))^{-1/2}(1-\del_m)(1+o(1))-1\\
&<&0
\end{eqnarray*}
for all large enough $m$. Therefore
\[
y_m>\exp\left(\frac{\log(m)}{2m}
(1-\del_m)\right)
\]
for all large enough $m$.

We have now proved upper and lower bounds on $y=y_m$ which together imply (\ref{sharpybound}). The proof of (\ref{sharplambound}) follows directly from the formula $\lam=\ir(y-y^{-1})$.\\
The locations of the other three solutions $y\in\R\backslash \{0\}$
are determined by using (\ref{r1r2formula}) and
(\ref{r1r2formula2}).

To prove the last statement of the lemma, we use again \eqref{eq:r2iy}, which for even $m$ becomes
\[
r_m^{(2)}(\ir y)= -(y-1)y^{2m+1}-(y+1).
\]
This immediately yields the bounds
\[
 r_m^{(2)}(\ir y)\le
 \begin{cases}
 -2y,\qquad&\text{if }0\le y<1,\\
 -1,\qquad&\text{if }y\ge 1,
 \end{cases}
 \]
 and so  $r_m^{(2)}(\ir y)< 0$ for $y>0$.
 As $r_m^{(2)}(0)=-1$, we deduce by Lemma \ref{lem:beta_factorisation} that $r_m^{(2)}(\ir y)=0$ has no solutions $y\in\R$.
\end{proof}

The last result of this section shows that the condition that \emph{both} $m$ and $n$ go to infinity is essential for the convergence of non-real eigenvalues to the real axis.

\begin{lemma}\label{lem:n=1}
Let $n=1$, $c=0$. Then there exists a sequence of purely imaginary eigenvalues $\lambda_m\in\spec{\cA_{m,1;0}}$ such that
\[
\liminf_{m\to\infty}\Im(\lambda_m)\ge\frac{9}{20}.
\]
\end{lemma}

\begin{proof} We act similarly to the proof of Lemma \ref{z=iybound} and consider
\[
\beta_{m,1}(\ir y,\ir y)=\ir^{m+1}y^{-3-m}\left((-1)^m y^{2m+4}-2(-1)^my^{2m+2}+2y^2-1\right)
\]
for $y\in\R\setminus\{0\}$. The $y$-zeros of $\beta_{m,1}(\ir y,\ir y)$ coincide with those of
\[
g_m(y):=(-1)^m y^{2m+4}-2(-1)^my^{2m+2}+2y^2-1.
\]
Consider the quantity
\[
g_m\left(\frac{5}{4}\right)g_m\left(\frac{3}{2}\right)=\left(\frac{17}{8}-7 (-1)^m 2^{-4(2+m)}5^{2+m}\right)\cdot\left(\frac{7}{2}+(-1)^m 9^{1+m} 2^{-2(2+m)}\right)
\]
It is easily checked that for $m>3$
\[
g_m\left(\frac{5}{4}\right)g_m\left(\frac{3}{2}\right)<0,
\]
and therefore there exists a zero $y_m\in\left(\frac{5}{4},\frac{3}{2}\right)$ of $g_m(y)$. Setting $\lambda_m=\ir y_m+(\ir y_m)^{-1}$, we have
\[
\Im(\lambda_m)=y_m-(y_m)^{-1}> \frac{5}{4}-\frac{4}{5}=\frac{9}{20}.
 \]
\end{proof}

\begin{remark} It is possible to show that for the sequence $\{y_m\}$ constructed in the proof of Lemma \ref{lem:n=1} one actually has $\lim_{m\to\infty} y_m=\sqrt{2}$ and so for the corresponding eigenvalues
$\lim_{m\to\infty}\lambda_m =\ir/ \sqrt{2}$. Indeed, re-write, for $y>1$,
\[
g_m(y)=(-1)^m\,y^{2m+2}(y+\sqrt{2})\,\left(y-\sqrt{2}+(-1)^m\,(2y^{-2m}-y^{-2m-2})\,(y+\sqrt{2})^{-1}\right).
\]
As $\{(-1)^m\,(2y^{-2m}-y^{-2m-2})\,(y+\sqrt{2})^{-1}\}$ converges to zero in $C^1$ for any sufficiently small  interval around $\sqrt{2}$ as $m\to\infty$, the result follows.
\end{remark}

\section{The case $c=0$, $m=n$}

In this section, we consider possibly the simplest case, namely $c=0$, $m=n$. We shall denote for simplicity
\begin{equation}\label{eq:A_moperatordef}
\cA_{m}:=\cA_{m,m;0}(\lambda)=H_{N}-\lambda D_{m,m},
\end{equation}
where we use the shorthand notation
\[
N=2m,\qquad H_{N}=H_{m,m;0}.
\]

Note that  $c=0$ implies
\begin{equation}\label{eq:z,c=0}
\lambda=z+\frac{1}{z}=w+\frac{1}{w},
\end{equation}
and we can take $w=z$.

In this case, the spectrum of the pencil has some additional symmetries.
\begin{lemma}\label{th:c0symm}
Let $n=m$, $c=0$. Then, in addition to the results of Theorem \ref{th:location_rough}, the spectrum $\spec{\cA_m}$ is invariant under
the symmetry $\lambda\to -\lambda$ and therefore $\spec{\cA_m}$ is invariant under reflections about the real and imaginary axes.
\end{lemma}

In fact, the same result holds even without the assumption $c=0$, see Lemma \ref{th:cne0symm} and its proof below.

Our main aim is to determine the asymptotic behaviour of the eigenvalues of $\cA_m$ for large $m$. We start with some numerical experiments. Figure \ref{fig:pict1} shows the location of
eigenvalues of $\cA_m$ on the complex plane for $m=100$, $250$, and $500$.

\begin{figure}[!thb]
\begin{center}
\includegraphics{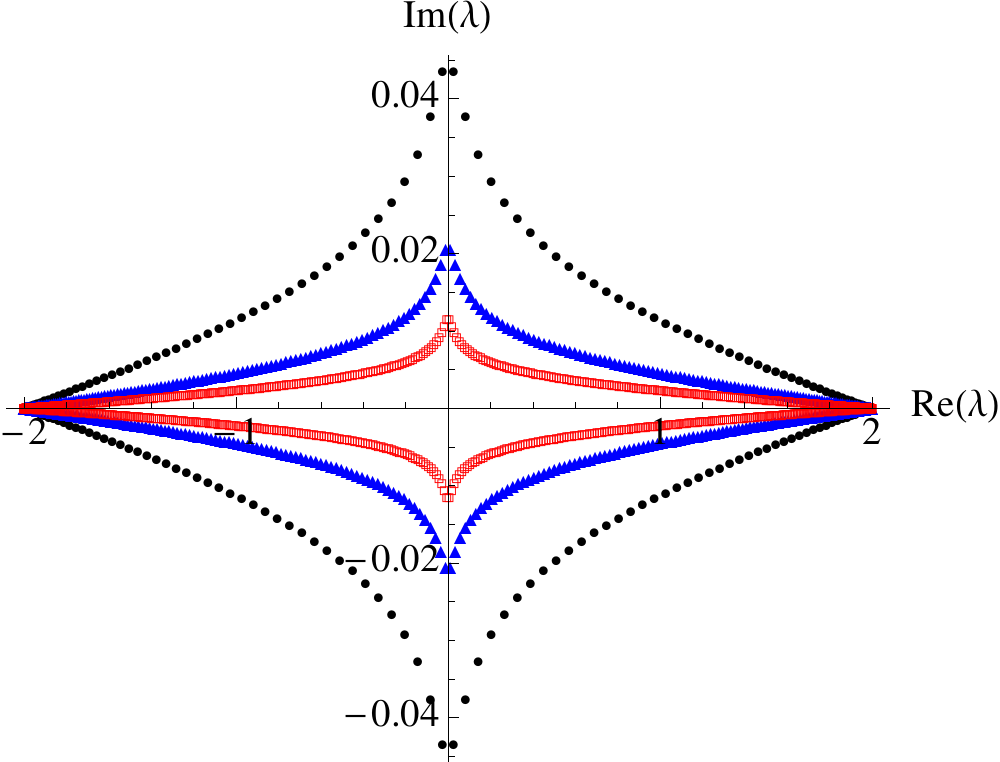}
\caption{\small $\spec{\cA_m}$ for $m=100$ (black circles), $m=250$ (blue triangles) and $m=500$ (red squares).\label{fig:pict1}}
\end{center}
\end{figure}

It is clear from Figure \ref{fig:pict1} that the eigenvalues in each case lie on certain curves in the complex plane, and that the shapes of these curves are somewhat similar. The situation becomes much clearer
if we write the eigenvalues of $\cA_m$ as $\lambda=u+\ir v/N=u+\ir v/(2m)$, where $u=\Re(\lambda)$ and $v=N\,\Im(\lambda)$, and redraw them in coordinates $(u,v)$, as in Figure~\ref{fig:pict2}.

\begin{figure}[!thb]
\begin{center}
\includegraphics{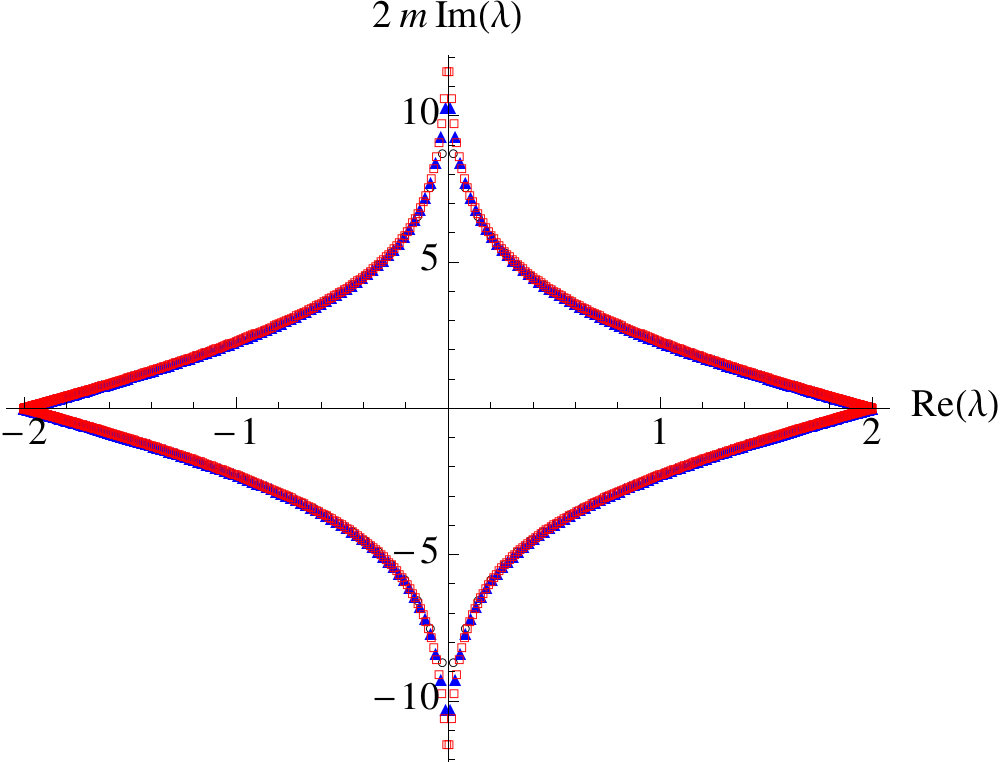}
\caption{\small $\spec{\cA_m}$ for $m=100$ (black circles), $m=250$ (blue triangles) and $m=500$ (red squares), drawn in coordinates $(\Re(\lambda), 2m\,\Im(\lambda))$.\label{fig:pict2}}
\end{center}
\end{figure}

In these coordinates, the spectra seem to lie (at least approximately) on the same curve, which is independent of $m$. We can in fact determine this curve explicitly.

\begin{theorem}\label{th:c=0}
Let $c=0$,  $n=m=N/2\to\infty$. The eigenvalues of $\cA_{m}$ are all non-real, and those not lying on the imaginary axis satisfy
\begin{equation}\label{eq:c=0asympt}
\Im(\lambda)=\pm  \frac{\Lambda_0(|\Re(\lambda)|)}{2m}+o(m^{-1}),
\end{equation}
\begin{equation}\label{eq:c=0Reasympt}
\Re(\lambda)=\pm 2\cos\left(\frac{2\pi k}{2m+1}\right)+o(m^{-1}),\qquad k=1,\dots,\left\lfloor \frac{m}{2}\right\rfloor,
\end{equation}
where $\lfloor\cdot\rfloor$ denotes the integer part, and
\begin{equation}\label{eq:Lambda0(u)}
\Lambda_0(u):=\sqrt{4-u^2}\log\left(\tan\left(\frac{\pi}{4}+\frac{1}{2}\arccos\left(\frac{u}{2}\right)\right)\right).
\end{equation}

If $m$ is even, there are no other eigenvalues.

If $m$ is odd, there are additionally two purely imaginary eigenvalues at
\begin{equation}\label{eq:imaginaries}
\lambda=\pm \ir \frac{\log(m)}{m} \left(1+o(1)\right).
\end{equation}
\end{theorem}

Before proving the theorem, we illustrate its effectiveness by some examples, see Figure \ref{fig:asympt}. Note the different behaviour, in the vicinity of the imaginary axis, for even and odd values of $m$.
\begin{figure}[!htb]
\begin{center}

\fbox{\includegraphics{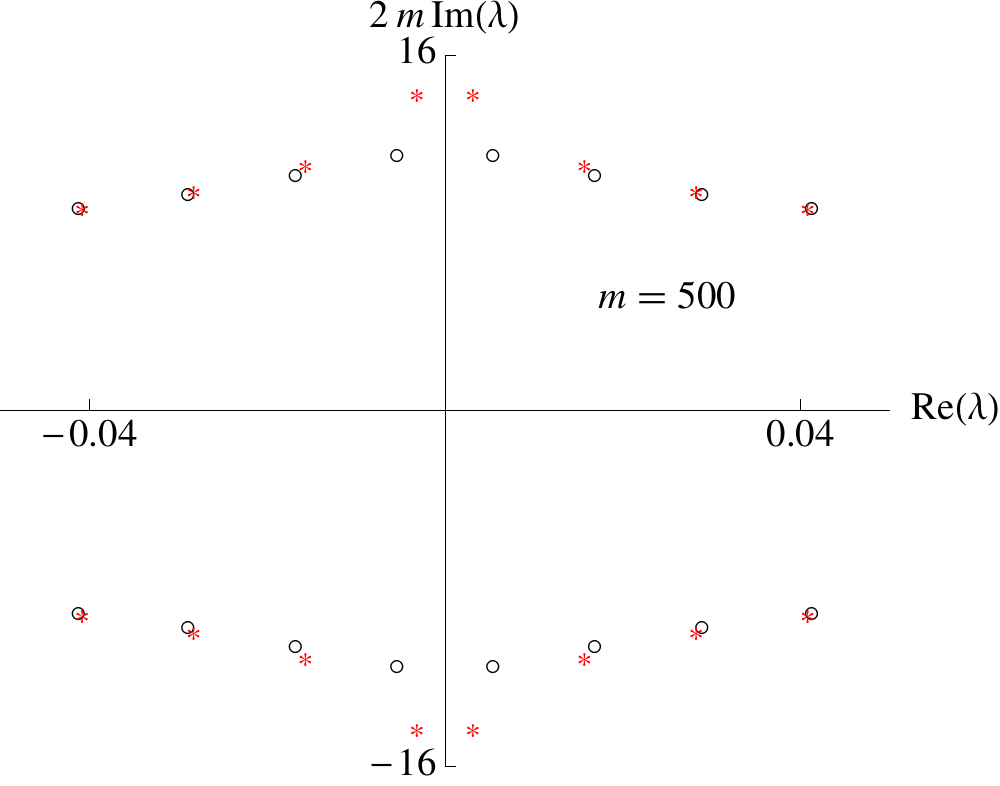}}\\

\

\fbox{\includegraphics{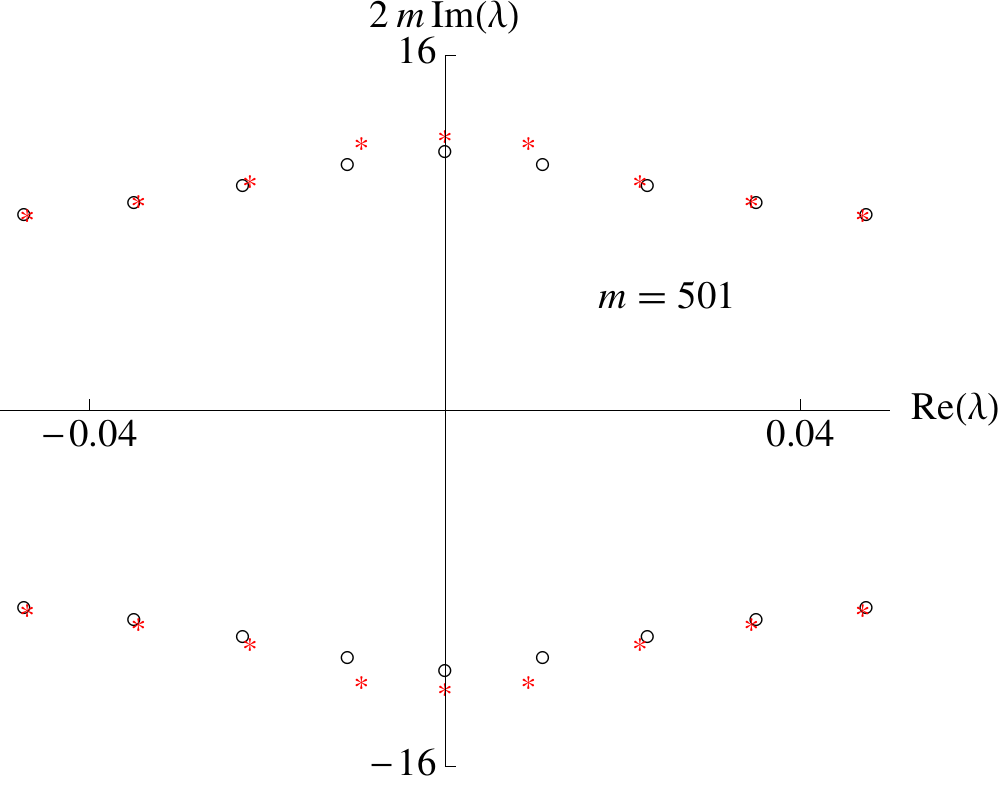}}
\caption{\small The blow-up near the imaginary axis of the numerically computed $\spec{\cA_{m}}$, as defined in \eqref{eq:A_moperatordef}, for $m=500$ and $m=501$ (white circles), drawn in coordinates
$(\Re(\lambda), 2m\,\Im(\lambda))$ together with their asymptotic values (red stars) given by Theorem \ref{th:c=0}.\label{fig:asympt}}
\end{center}
\end{figure}

\begin{proof}[Proof of Theorem \ref{th:c=0}]
This relies heavily on Theorem \ref{th:main1} and Lemma \ref{lem:beta_factorisation}.

We first prove that all eigenvalues are non-real. Assume the opposite, and consider a real eigenvalue $\lambda$. By Theorem \ref{th:main1}(b), either the corresponding value of $z$ is real, or it lies on the unit circle and satisfies
\eqref{qtildemmz}. Suppose that $r_m^{(2)}(z)=0$. From \eqref{r2def}, we have
\begin{equation}\label{eq:r2eq}
z^{2m+1}=-\ir\frac{z+\ir}{z-\ir}.
\end{equation}
and therefore, by taking absolute values,
\begin{equation}\label{eq:r2eqmodulus}
|z|^{2m+1}=\frac{|z+\ir|}{|z-\ir|}.
\end{equation}
The left-hand side of \eqref{eq:r2eqmodulus} is equal to one if and only if $z$ lies on the unit circle, and the right-hand side is equal to one if and only if $z$ is equidistant from $-\ir$ and $+\ir$, and is therefore real. Thus, the only possibilities for a solution of \eqref{eq:r2eqmodulus} which is either real or lying on the unit circle are $z=\pm1$. But these solutions correspond to $\lambda=\pm 2$, which contradicts Lemma \ref{th:location_rough}(b). The case
$r_m^{(1)}(z)=0$ is similar. Thus, $\spec{\cA_m}\cap\R=\varnothing$.

To obtain the asymptotics \eqref{eq:c=0asympt}--\eqref{eq:Lambda0(u)} formally we once more look at the solutions of the equations $r^{(j)}_m(z)=0$, $j=1,2$, now taking $m\to\infty$. Consider again the case $j=2$, i.e.\ the equation
\eqref{eq:r2eq}. We seek solutions of the form
\begin{equation}\label{eq:zanzats1}
z=\er^{\ir\theta+s/m},
 \end{equation}
 where
\begin{equation}\label{eq:sanzats1}
 s=\sum_{k=0}^\infty \frac{s_k}{m^k},
\end{equation}
and $\theta$ and $s_k$ are real.
We justify making this ansatz at the end of the proof. We recall \eqref{eq:z,c=0}
and normalise the eigenvalues as
\begin{equation}\label{eq:lamscaling}
 \lambda=z+\frac{1}{z}=u+\frac{\ir v}{2m},\qquad u,v\in\R
\end{equation}
By separating the real and imaginary parts, \eqref{eq:zanzats1}, \eqref{eq:lamscaling} immediately imply that
 \begin{equation}\label{eq:ufull}
  u=2\cos(\theta)\cosh(s/m)
  \end{equation}
  and
   \begin{equation}\label{eq:vfull}
\frac{v}{2m}=2\sin(\theta)\sinh(s/m).
  \end{equation}

We may also assume the first of inequalities \eqref{eq:z,w>1}, and look, in the first instance, for the eigenvalues $\lambda$ in the first quadrant, so that $u,v>0$. Thus, $s>0$, and $0<\theta<\pi/2$.

We first use \eqref{eq:r2eqmodulus}, which may be rewritten as
\begin{equation}\label{eq:r2eqmodulus1}
\begin{split}
\er^{(2m+1)s/m}&=\left|\frac{z+\ir}{z-\ir}\right|=\left| \frac{\sqrt{\ir z}-\sqrt{\ir z}^{\,-1}}{\sqrt{\ir z}+\sqrt{\ir z}^{\,-1}} \right| \\
&=\left|\tan\left(\frac{\log\left(\sqrt{\ir z}\right)}{\ir}\right)\right|\\
&=\left|\tan\left(\frac{\pi}{4}+\frac{\theta}{2}-\ir\frac{s}{2m}\right)\right|,
\end{split}
\end{equation}
where all the square roots are understood in the principal value sense.

Retaining only the terms of order $O(1)$ for the real parts, and the terms of order $O\!\left(m^{-1}\right)$ for the imaginary parts in \eqref{eq:ufull}, \eqref{eq:vfull}, then using \eqref{eq:sanzats1}, and substituting the results
\[
u=2\cos(\theta)+o(1),\qquad v=4s_0\sin(\theta)+o(1)=2s_0\sqrt{4-u^2}+o(1)
\]
into \eqref{eq:r2eqmodulus1}, we obtain
\[
\theta=\arccos\left(\frac{u}{2}\right),\qquad \er^{2s_0(\theta)}=\tan\left(\frac{\pi}{4}+\frac{\theta}{2}\right).
\]
up to leading order.

The results \eqref{eq:c=0asympt}, \eqref{eq:Lambda0(u)} (for $\lambda$ in the first quadrant) now follow from backward substitutions.
We extend it to the other quadrants by using the symmetries described in Lemma \ref{th:c0symm}.

To prove \eqref{eq:c=0Reasympt}, we consider the real parts of the equation \eqref{eq:r2eq}, which leads, again to leading order, to
\[
\cos((2m+1)\theta)\tan\left(\frac{\pi}{4}+\frac{\theta}{2}\right)=\frac{\cos(\theta)}{1-\sin(\theta)}.
\]
This is equivalent to
\[
\cos((2m+1)\theta)=1,
\]
which implies \eqref{eq:c=0Reasympt} once we restrict ourselves to the first quadrant; we again extend to other quadrants by symmetry.

The final statement of the theorem is just a re-statement of Lemma \ref{z=iybound}.

To locate the solutions of the equation $\beta_{m,m}(z)=0$ rigorously for large $m$, where $\beta_{m,m}$ is analytic in $z$, we now should proceed to an application of Rouche's theorem and/or a contraction mapping theorem to prove that the zeros are indeed where expected from the above non-rigorous asymptotic analysis.
This is very tedious, because it depends on obtaining accurate estimates of  $\beta_{m,m}(z)$ for all $z$ in small closed circles around the expected positions of the zeros; such calculations are carried out for different models in \cite[section~4]{EBD1}, \cite[section~2]{EBD2} and \cite{DjMi}.

Once this has been carried out, the ansatz used earlier in the proof is justified by observing that if $m$ is even then \eqref{eq:c=0Reasympt} yields $N$ distinct eigenvalues and $\cA_m$ is an $N\times N$ matrix, so there cannot be any eigenvalues having a different asymptotic form. If $m$ is odd then \eqref{eq:c=0Reasympt} yields $N-2$ distinct eigenvalues, but \eqref{eq:imaginaries} provides another $2$ eigenvalues, so once again there cannot be any eigenvalues having a different asymptotic form. We omit this part of the proof in order to focus on the more interesting aspects of the analysis.
\end{proof}

\begin{remark}
One can determine further asymptotic terms in  \eqref{eq:sanzats1} and therefore in  \eqref{eq:c=0asympt}, \eqref{eq:c=0Reasympt}, by continuing the iteration process: on the next step, determine $s_1$ by retaining the terms of order $O\!\left(m^{-1}\right)$ in \eqref{eq:r2eqmodulus1}, and then use the real parts of the equation \eqref{eq:r2eq} to find a correction to $\theta$, and so on.
\end{remark}

\begin{remark} As $u\to\pm 2$, $\Lambda_0(u)\sim (2-|u|)$. More interestingly, $\Lambda_0(u)$ blows up logarithmically as $u\to0$. This matches the behaviour which we have seen in Lemma~\ref{z=iybound} for purely imaginary eigenvalues and is also compatible with the numerical calculations displayed in Figure~\ref{fig:pict2}.
\end{remark}

\section{The case $c\ne 0$, $m=n$}

In this section, we denote for brevity
\[
\cA_{m;c}=\cA_{m;c}(\lambda):=\cA_{m,m;c}(\lambda)=H_{N;c}-\lambda D_{m,m},
\]
where
\[
N=2m,\qquad H_{N;c}=H_{m,m;c}.
\]

We shall see later that the spectral ``picture'' of the pencil $\cA_{m;c}$ with $c\ne 0$ may be quite different from that of $\cA_m$.
Nevertheless, the condition $m=n$ implies extra symmetries as in Lemma \ref{th:c0symm}.

\begin{lemma}\label{th:cne0symm}
Let $n=m$, $c\ne 0$. Then, in addition to the results of Theorem \ref{th:location_rough}, the spectrum $\spec{\cA_{m;c}}$ is invariant under
the symmetry $\lambda\to -\lambda$ and therefore $\spec{\cA_{m;c}}$ is invariant under reflections about the real and imaginary axes. Moreover, $\spec{\cA_{m;c}}=\spec{\cA_{m;-c}}$.
\end{lemma}

\begin{proof} Both results follow immediately from Theorem \ref{th:main1}, the explicit formula \eqref{betdef}, and substitutions \eqref{zw_subs}.  Indeed, the symmetry $\lambda\to-\lambda$ corresponds to the symmetry
$\{z,w\}\to\{-w,-z\}$, and we have $\beta_{m,m}(-w,-z)=\beta_{m,m}(z,w)$. The change  $c\to-c$ corresponds to $z\leftrightarrow w$, and we also have $\beta_{m,m}(w,z)=\beta_{m,m}(z,w)$.
\end{proof}

\begin{remark}
In general, for $m\ne n$, and $c\ne 0$, the spectrum $\spec{\cA_{m,n;c}}$ is \emph{not} symmetric with respect to  either $\lambda\to -\lambda$ or $c\to -c$.
\end{remark}

Since we know by Lemma \ref{th:location_rough}(c) that all the eigenvalues are real when $|c|\ge 2$, we may consider, in our study of non-real eigenvalues, only the case $0<c<2$. We also already know by Lemma \ref{th:location_rough}(b) that all the eigenvalues of $\cA_{m;c}$ satisfy $|\Re(\lambda)|\leq |\lambda |<2+|c|$.

We hope to improve this estimate for the real parts of the non-real eigenvalues. The following conjecture is amply confirmed by numerical and asymptotic calculations, for \emph{every} value of $m$, but its proof seems to be much harder than one would expect. We defer this to a separate paper, because we wish to concentrate on the location of the complex eigenvalues. Strictly speaking, we do not use the conjecture elsewhere.

\begin{conjecture}\label{conj:Relambdabound}
Let $c>0$. If $\lambda$ is a non-real eigenvalue of $\cA_{m;c}$, then $c<2$ and
\begin{equation}
|\lambda\pm c|<2,
\end{equation}
and therefore
\begin{equation}
|\Re(\lambda)|\le 2-c.
\end{equation}
\end{conjecture}

Before proceeding to the asymptotic analysis of $\spec{\cA_{m;c}}$, we start, as in the case $c=0$, with some numerical experiments. The choice $c=\sqrt{5}/2$ in Figure~\ref{fig:pict3} is intended to yield generic behaviour.

\begin{figure}[!htb]
\begin{center}
\fbox{\includegraphics{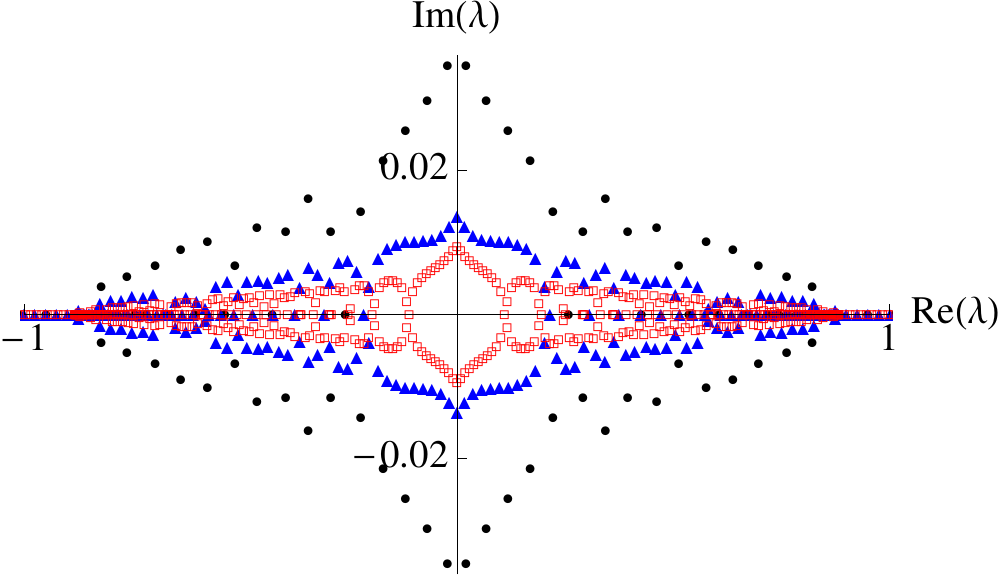}}
\caption{\small $\spec{\cA_{m;c}}$ for $c=\sqrt{5}/2$ and $m=100$ (black circles), $m=250$ (blue triangles) and $m=500$ (red squares). Some real eigenvalues with $2-c<|\lambda|<2+c$ are not shown. \label{fig:pict3}}
\end{center}
\end{figure}

One can see that for $c\ne 0$ the spectral picture is more complicated: there are both real and non-real eigenvalues, and it is implausible that the eigenvalues for a fixed $c$ lie on a particular curve even after scaling in the vertical direction. However, one observes some common features for all values of $c$ if one superimposes the spectra of $\spec{\cA_{m;c}}$, with the imaginary parts scaled by $2m$ as before; see Figure \ref{fig:pict4}.

\begin{figure}[!htb]
\begin{center}
\fbox{\includegraphics{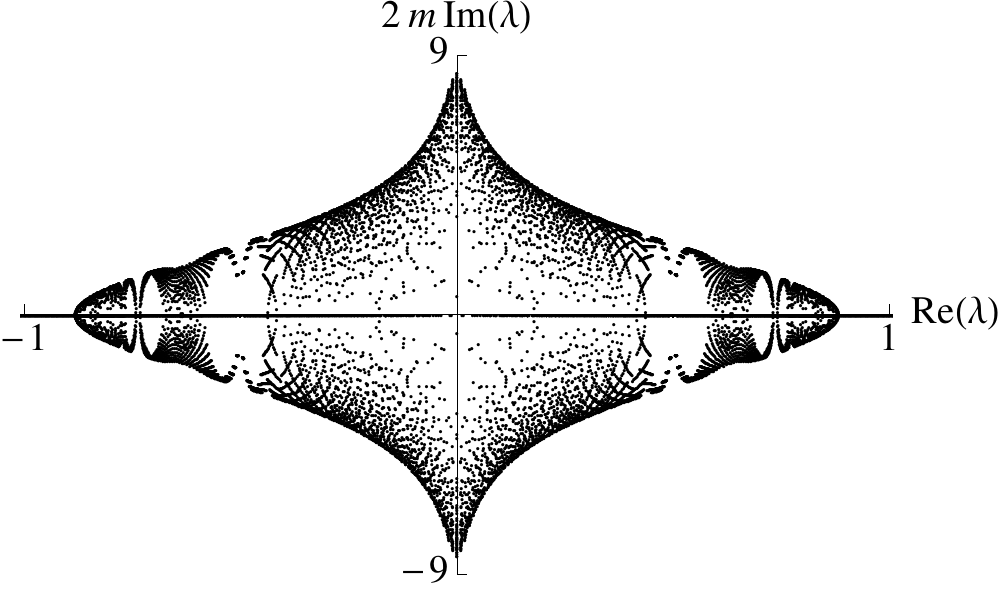}}
%
%
%
%
%
%
\caption{\small $\bigcup_{m=150}^{250}\spec{\cA_{m;c}}$ for $c=\sqrt{5}/2$, using the coordinates $(\Re(\lambda), 2m\,\Im(\lambda))$.  Some real eigenvalues with $2-c<|\lambda|<2+c$ are not shown. \label{fig:pict4}}
\end{center}
\end{figure}

We see that the spectra have a common bounding curve, although the behaviour in the interior differs for different $c$. In the rest of this section we deduce an explicit expression for the bounding curve. Our main theorem is as follows. We discuss its optimality at the end of the section.

\begin{theorem}\label{th:asymptcne0} Let $0<c<2$, $n=m=N/2\to\infty$. The eigenvalues of $\cA_{m;c}$ that satisfy
$0<|\Re(\lambda)|< 2-c$ also satisfy
\[
|\Im(\lambda)|\le \frac{\Lambda_c(|\Re(\lambda)|)}{2m}+o\left(m^{-1}\right),
\]
where
\[
\Lambda_c(u):=X_{c,u}^{-1}\left(\tan\left(\frac{1}{2}\arccos\left(\frac{u-c}{2}\right)\right)\tan\left(\frac{1}{2}\arccos\left(\frac{u+c}{2}\right)\right)\right),
\]
and $X_{c,u}^{-1}$ is the inverse of the monotonic increasing analytic function $X_{c,u}:(0,\infty)\to (0,1)$ defined by
\[
 X_{c,u}(v):=\tanh\left(\frac{v}{2\sqrt{4-(u-c)^2}}\right)\tanh\left(\frac{v}{2\sqrt{4-(u+c)^2}}\right).
 \]
\end{theorem}

\begin{remark} The hypothesis of Theorem \ref{th:asymptcne0} carefully avoids the need to use Conjecture \ref{conj:Relambdabound}.
\end{remark}

The proof depends on some classical facts about fractional linear transformations on the complex plane. Let $C(a,\rho)=\{\zeta\in\C:|\zeta-a|=\rho\}$ denote the circle in the complex plane with centre $a$ and radius $\rho$.

\begin{lemma}\label{lem:circle} Let $\xi\in\C\setminus\{0\}$, and let $\kappa>0$, $\kappa\ne 1$. The set
\[
S_{\xi, \kappa}:=\left\{\zeta\in\C: \left|\zeta-\xi^{-1}\right|=\kappa|\zeta-\xi|\right\}
\]
is the circle $C\!\left(a_{\xi, \kappa}, \rho_{\xi, \kappa}\right)$ with centre and radius  given by
\[
a_{\xi, \kappa}=\frac{\kappa\xi-\frac{1}{\kappa\xi}}{\kappa-\frac{1}{\kappa}},\qquad \rho_{\xi, \kappa}=\left|\frac{\xi-\frac{1}{\xi}}{\kappa-\frac{1}{\kappa}}\right|.
\]
\end{lemma}

\begin{proof}[Proof of Lemma \ref{lem:circle}]
Let  $\kappa|\zeta-\xi|=\left|\zeta-\xi^{-1}\right|$. Taking squares on both sides, expanding and collecting, we have
\[
\begin{split}
\kappa^2|\zeta|^2-2\kappa^2\Re\left(\zeta\overline{\xi}\right)+\kappa^2|\xi|^2&=|\zeta|^2-2\Re\left(\zeta\overline{\xi^{-1}}\right)+|\xi|^{-2}\\
&\Updownarrow\\
|\zeta|^2-2\Re\left(\zeta\cdot\frac{\overline{\kappa^2\xi-\zeta^{-1}}}{\kappa^2-1}\right)&=\frac{|\xi|^{-2}-\kappa^2|\xi|^2}{\kappa^2-1}\\
&\Updownarrow\\
\left|\zeta-\frac{\kappa\xi-\frac{1}{\kappa\xi}}{\kappa-\frac{1}{\kappa}}\right|^2&=\frac{\left|\kappa^2\xi-\zeta^{-1}\right|^2}{(\kappa^2-1)^2}+\frac{|\xi|^{-2}-\kappa^2|\xi|^2}{\kappa^2-1}\\
&=\frac{\left|\xi-\frac{1}{\xi}\right|^2}{\left(\kappa-\frac{1}{\kappa}\right)^2},
\end{split}
\]
which implies the result.
\end{proof}

\begin{proof}[Proof of Theorem \ref{th:asymptcne0}] As in the case $c=0$, we rescale the eigenvalues of $\cA_{m;c}$
by means of the formula $\lambda=u+\ir v/(2m)$ as in \eqref{eq:lamscaling}. We restrict ourselves to the first quadrant
$u,v>0$, because bounds in the other quadrants can be obtained by symmetry.
The hypotheses of the theorem imply that $0<u<2-c$. The strategy is to take  $c\in(0,2)$ and $u$ to be fixed, and to find necessary conditions on $v$ for $\lambda$ to be an eigenvalue.

We start our formal asymptotic procedure by assuming the ansatzes
\begin{equation}\label{eq:zansatz2}
z=\er^{\ir\theta+s/m},
\end{equation}
 \begin{equation}\label{eq:wansatz2}
w=\er^{\ir\phi+t/m},
\end{equation}
where
\begin{equation}\label{eq:sansatz2}
 s=\sum_{k=0}^\infty \frac{s_k}{m^k},
 \end{equation}
 and
\begin{equation}\label{eq:tansatz2}
 t=\sum_{k=0}^\infty \frac{t_k}{m^k},
 \end{equation}
 together with \eqref{zw_subs}. By \eqref{eq:z,w>1}  and conditions on $u,v$ we can assume $s,t>0$ and $0<\theta,\phi<\pi$.

By separating the real and imaginary parts in the first relation \eqref{zw_subs} using \eqref{eq:zansatz2}  we get
\begin{align*}
2\cosh(s/m)\cos(\theta)&=u-c,\\
2\sinh(s/m)\sin(\theta)&=v/(2m).
\end{align*}
Taking account of \eqref{eq:sansatz2}, these yield
\begin{align*}
\cos(\theta)&\sim(u-c)/2\in (-1,1),\\
\sin(\theta)&\sim v/(4s_0)\in (0,1],
\end{align*}
to leading order in $m^{-1}$. From now on, we shall write these and similar asymptotic equalities as if they were exact identities. The first equation determines $\theta$ and the second then determines $s_0$. Indeed, we obtain
\begin{align}
\theta&= \arccos((u-c)/2)\in (0,\pi), \label{thetaformula}\\
s_0&= \frac{v}{4\sin(\theta)}=\frac{v}{2\sqrt{4-(u-c)^2}}>0 .\label{sformula}
\end{align}

Treating similarly the second relation in \eqref{zw_subs}, and using \eqref{eq:wansatz2} and \eqref{eq:tansatz2}, we obtain
\begin{align*}
2\cosh(t/m)\cos(\phi)&=u+c,\\
2\sinh(t/m)\sin(\phi)&=v/(2m),
\end{align*}
then to leading order in $m^{-1}$ one has
\begin{align*}
\cos(\phi)&=(u+c)/2\in (-1,1),\\
\sin(\phi)&= v/(4t_0)\in (0,1],
\end{align*}
which yield
\begin{align}
\phi&= \arccos((u+c)/2)\in (0,\pi), \label{phiformula}\\
t_0&= \frac{v}{4\sin(\phi)}=\frac{v}{2\sqrt{4-(u+c)^2}}>0, \label{tformula}
\end{align}
for $u\in(0,2-c)$, as assumed in the theorem.

If $\lambda$ is a non-real eigenvalue of $\cA_{m;c}$ then $z$ and $w$ should satisfy \eqref{F_condition}, namely $F_m(z)F_m(w)=-1$.
Our strategy will be to consider $u\in(0,2-c)$ fixed and find necessary conditions on $v$ for \eqref{F_condition} to hold. We re-write \eqref{F_condition}
as
\begin{align}
F_m(z)&= \zeta,\label{zetaformula1}\\
F_m(w)&= -1/\zeta, \label{zetaformula2}
\end{align}
with $\zeta\in\C$. One may rewrite \eqref{zetaformula1} in the form
\[
z^{2m}=\frac{z^{-1}-\zeta}{z-\zeta}.
\]
Taking absolute values and using \eqref{eq:zansatz2} and  \eqref{eq:sansatz2},  we obtain
\[
\rme^{2s_0}=\left| \frac{\rme^{-\ir\theta}-\zeta}{\rme^{\ir\theta}-\zeta}\right|,
\]
up to a correction that is $o(1)$ as $m\to\infty$.
Applying Lemma~\ref{lem:circle}, we deduce that $\zeta\in C(a_1,\rho_1)$, where
\begin{equation}\label{eq:circle1}
\begin{split}
a_1&=\frac{\er^{2s_0+\ir\theta}-\er^{-2s_0-\ir\theta}}{\er^{2s_0}-\er^{-2s_0}}=\cos(\theta)+\ir\sin(\theta)\coth(2s_0), \\
 \rho_1&=\left|\frac{\er^{\ir\theta}-\er^{-\ir\theta}}{\er^{2s_0}-\er^{-2s_0}}\right|=\frac{\sin(\theta)}{\sinh(2s_0)}.
\end{split}
\end{equation}

We deal in a similar manner with the equation \eqref{zetaformula1}, rewriting it as
\[
w^{2m}=\frac{w^{-1}+\zeta^{-1}}{w+\zeta}.
\]
Again taking absolute values, and using \eqref{eq:wansatz2} and \eqref{eq:tansatz2}, we obtain
\[
\er^{t_0}=\left| \frac{\rme^{\ir\phi}+\zeta}{\rme^{-\ir\phi}+\zeta}\right|,
\]
with a correction that is $o(1)$. This yields $\zeta\in C(a_2,\rho_2)$, where
\begin{equation}\label{eq:circle2}
\begin{split}
a_2&=\frac{-\er^{2t_0-\ir\phi}+\er^{-2t_0+\ir\phi}}{\er^{2t_0}-\er^{-2t_0}}=-\cos(\phi)+\ir\sin(\phi)\coth(2t_0),\\
\rho_2&=\left|\frac{-\er^{-\ir\phi}+\er^{\ir\phi}}{\er^{2t_0}-\er^{-2t_0}}\right|=\frac{\sin(\phi)}{\sinh(2t_0)}.
\end{split}
\end{equation}

In order to obtain a solution of the system \eqref{zetaformula1}, \eqref{zetaformula2}, the circles \eqref{eq:circle1} and \eqref{eq:circle2} must intersect. Thus,
we get
\begin{equation}\label{eq:intersect1}
|a_1-a_2|^2 \le (\rho_1+\rho_2)^2.
\end{equation}
Substituting the explicit expressions for the centres and radii of the circles, we obtain after some simplifications,
\begin{equation}\label{eq:intersect2}
\frac{2+2\cos(\theta)\cos(\phi)}{\sin(\theta)\sin(\phi)}\le \tanh(s_0)\tanh(t_0)+\coth(s_0)\coth(t_0).
\end{equation}
We substitute
\begin{equation}\label{eq:intersect3}
X=\tanh(s_0)\tanh(t_0),\qquad Z=\frac{2+2\cos(\theta)\cos(\phi)}{\sin(\theta)\sin(\phi)},
\end{equation}
into \eqref{eq:intersect2},
noting that $0<X<1$ and $Z>2$.
This yields the inequality $X+1/X\ge Z$ and then
\begin{equation}\label{eq:intersect4}
X\le \frac{Z-\sqrt{Z^2-4}}{2}.
\end{equation}
(We can ignore the other interval in the solution of the inequality because $X\le 1$.) With the help of some trigonometry, the inequality \eqref{eq:intersect4} can be simplified to
\begin{equation}\label{eq:intersect5}
\tanh(s_0)\tanh(t_0)\le \tan\left(\frac{\theta}{2}\right)\tan\left(\frac{\phi}{2}\right).
\end{equation}

We now substitute the formulae \eqref{thetaformula}--\eqref{tformula}  into \eqref{eq:intersect5} to obtain
\begin{equation}\label{eq:intersect6}
\begin{split}
&\tanh\left(\frac{v}{2\sqrt{4-(u-c)^2}}\right)\tanh\left(\frac{v}{2\sqrt{4-(u+c)^2}}\right)\\
&\quad \le \tan\left(\frac{1}{2}\arccos\left(\frac{u-c}{2}\right)\right)\tan\left(\frac{1}{2}\arccos\left(\frac{u+c}{2}\right)\right)
\end{split}
\end{equation}

The left-hand side of \eqref{eq:intersect6} is a monotone-increasing function of $v$, so it has an inverse. The statement of the theorem now follows by extending the result by symmetry to other quadrant.
\end{proof}

We illustrate the results in Figure \ref{fig:pict6}. It appears from the figure that the bounding curve found in Theorem~\ref{th:asymptcne0} is optimal, but we have no proof of this. It also seems plausible that for every choice of $c\in(0,2)$ the union of the non-real spectra of $\cA_{m;c}$ over all $m$ is dense in the region inside the bounding curve. Further computations suggest that the apparent gap around the point $0$ or even around the imaginary axis is filled at a logarithmic rate as $m$ increases. We investigate this further in Section~\ref{sec:7}. We also produced a video showing how the spectrum changes with $c$ for fixed $m$ (see Appendix), from which further very complex structure is apparent. This involves number-theoretic properties of $\arccos(c/2)$; see Section \ref{sec:7}. We hope to investigate these in a later paper.

\begin{figure}[!htb]
\begin{center}
\fbox{\includegraphics{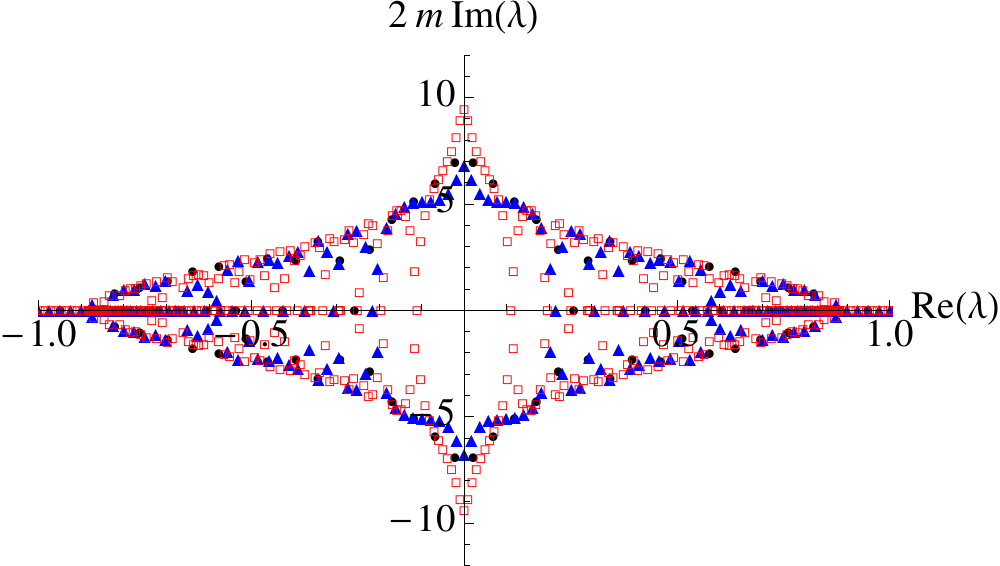}}
\caption{\small $\spec{\cA_{m;c}}$ for $c=\sqrt{5}/2$, and $m=100$ (black circles), $m=250$ (blue triangles) and $m=500$ (red squares),
drawn in coordinates $(\Re(\lambda), 2m\,\Im(\lambda))$ together with the graphs of  the functions $\pm \Lambda_c(\Re(\lambda))$ (solid black lines).  Some real eigenvalues with $2-c<|\lambda|<2+c$ are not shown. \label{fig:pict6}}
\end{center}
\end{figure}

\section{The case $c=0$, $m\ne n$}

If $\ell\in\N\cup\{0\}$ is fixed and $n=m+\ell$, then the matrices have size $N\times N$ where $N=2m+\ell$. The following theorem shows that the imaginary parts of the eigenvalues of the pencil $\cA_{m,n}$ depend very sensitively on the precise values of $m$ and $\ell$ by comparing the cases $\ell=0$ and $\ell=1$.
We define the algebraic multiplicity of an eigenvalue of a pencil to be the order of the zero of the corresponding determinant; see \eqref{eq:pdet}.

\begin{theorem}\label{n_is_m_plus_1_theorem}
If $n=m+1$, then every eigenvalue of $\cA_{m,n}$ is real and has geometric multiplicity $1$. The non-zero eigenvalues all have algebraic multiplicity $2$. The eigenvalue $0$ has algebraic multiplicity $1$ if $m$ is even and algebraic multiplicity $3$ if $m$ is odd.
\end{theorem}

\begin{proof}
If one rewrites the eigenvalue equation as a second order recurrence relation, it is immediate that every eigenvalue has geometric multiplicity $1$.

If one expands
\begin{equation}
p_{m,m+1}(\lam)=\det(H_{2m+1}-\lam D_{m,m+1})\label{eq:pdet}
\end{equation}
along the $(m+1)$th row, one obtains
\begin{equation}
p_{m,m+1}(\lam)=-\lam\, q_m(\lam) q_m(-\lam)=
(-1)^{m+1}\lam\, q_m(\lam)^2
\label{n=m+1det}
\end{equation}
where
\[
q_m(\lam)=\det(H_m-\lam I_m).
\]
The roots of $q_m$ are the eigenvalues of $H_m$, given by $\lam_r=2\cos(\pi r/(m+1))$, where $1\leq r\leq m$, cf. \eqref{Heig}. Each eigenvalue of $H_m$ is real and has algebraic multiplicity $1$. Equation \eqref{n=m+1det} now implies each non-zero eigenvalue of $\cA_{m,m+1}$ is real and has algebraic multiplicity $2$. The algebraic multiplicity of $0$ depends on whether $q_m(0)$ vanishes; this happens if and only if $m$ is odd.
\end{proof}

\begin{example}\label{7times7example}
If $m=3$ and $n=4$ then $N=7$, and a simple calculation shows that $\cA_{3,4}$ has three distinct eigenvalues, namely $0$ and $\pm \sqrt{2}$, each of which has geometric multiplicity $1$. The eigenvalues $\pm \sqrt{2}$ both have algebraic multiplicity $2$ and $0$ has algebraic multiplicity $3$. These facts are all reflected in the non-trivial Jordan form of
$D^{-1}H$.\\
\end{example}

\section{The eigenvalue $0$ and small eigenvalues}\label{sec:7}

In this section we present some preliminary results about the part of the spectrum of $\cA_{m;c}$ that is close to $0$. The following is our main result.

\begin{theorem}\label{zerospectrum}
If $0\le c<\infty$ then $0\in \spec{\cA_{m,n;c}}$ if and only if
\begin{equation}\label{cvaluesfor0}
c=2\cos(\pi j/(N+1))
\end{equation}
where the integer $j$ satisfies $1\leq j\leq N=m+n$; this implies that $0<c<2$. If $0\notin \spec{\cA_{m,n;c}}$ then
\begin{equation}
d_{m,n;c}:=\dist(\spec{\cA_{m,n;c}},0)\geq \del_{N;c},\label{resolvnorm}
\end{equation}
where 
\begin{equation}
\del_{N;c}:=\min_{1\leq j\leq N}
|c-2\cos(\pi j/(N+1))|.\label{def:delNc}
\end{equation}
\end{theorem}

\begin{proof}
The proof essentially continues the proof of Lemma \ref{th:location_rough}(b). If one puts $A_{m,n;c}=D_{m,n}^{-1}(H_{N}+cI)$, then $\spec{\cA_{m,n;c}}=\spec{A_{m,n;c}}$. Since $D_{m,n}$ is invertible, $0\in \Spec(A_{m,n;c})$ if and only if $H_{N}+cI$ is not invertible, or equivalently if and only if $-c\in\Spec(H_{N})$. The spectrum of $H_{N}$ is given by \eqref{Heig}  and yields the first statement of the theorem.

If $0\notin \spec{\cA_{m,n;c}}$ then the invertibility of $A_{m,n;c}$ and
$\norm D_{m,n}^{\pm 1}\norm =1$ together yield
\[
\norm A_{m,n;c}^{-1}\norm^{-1}=
\norm (H_{N}+cI)^{-1}\norm^{-1}=\del_{N;c};
\]
this uses the self-adjointness of $H_{N}$ and its known spectrum \eqref{Heig}.
If $|\lam |<\del_{N;c}$ then
\[
\begin{split}
\norm (\lam I-A_{m,n;c})^{-1}\norm
&= \norm A_{m,n;c}^{-1}(I-\lam A_{m,n;c}^{-1})^{-1}\norm\\
&\leq \norm A_{m,n;c}^{-1}\norm\sum_{k=0}^\infty |\lam/\del_{N;c}|^k\\
&<\infty.
\end{split}
\]
Therefore $\lam\notin \spec{\cA_{m,n;c}}$.
\end{proof}

Theorem \ref{zerospectrum} indicates that the rate at which the smallest absolute value eigenvalue converges to zero as
$N$ increases is determined by the Diophantine properties of $\pi^{-1} \arccos(c/2)$.  Numerical experiments indicate that the actual behaviour of  $d_{m,n;c}$  is similar to
$\del_{N;c}$, as shown in Figures \ref{fig:pict7} and \ref{fig:pict8}.

\begin{figure}[!htb]
\begin{center}
\fbox{\includegraphics{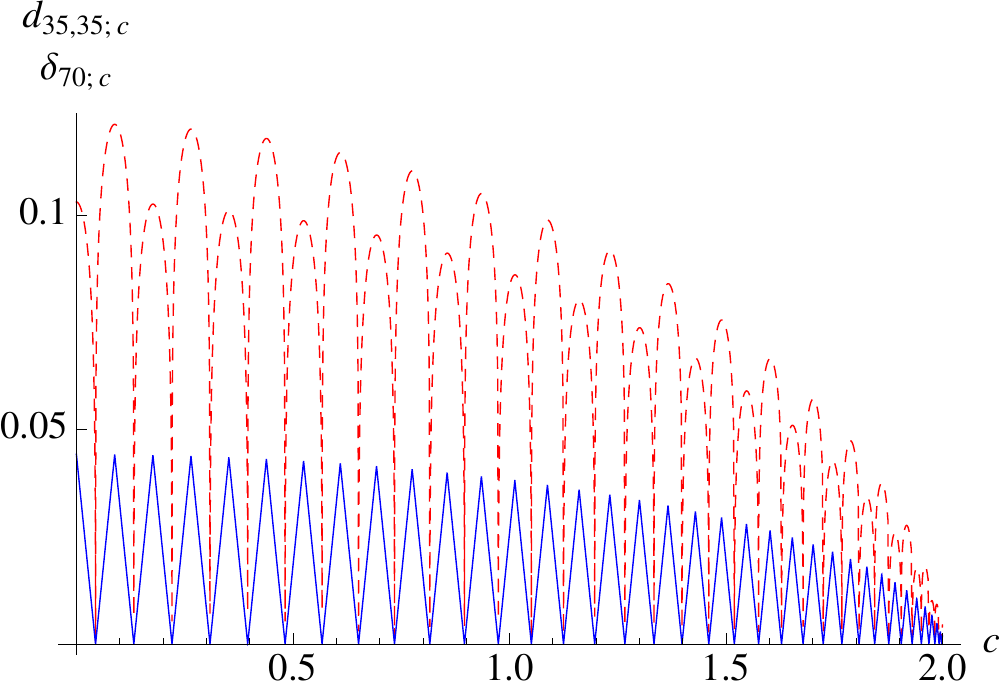}}
\caption{\small $d_{35,35;c}$ (dashed red line) and $\del_{70;c}$ (solid blue line), drawn as functions of $c$. Note the values \eqref{cvaluesfor0} (with $N=70$) where both functions vanish. \label{fig:pict7}}
\end{center}
\end{figure}

\begin{figure}[!htb]
\begin{center}
\fbox{\includegraphics{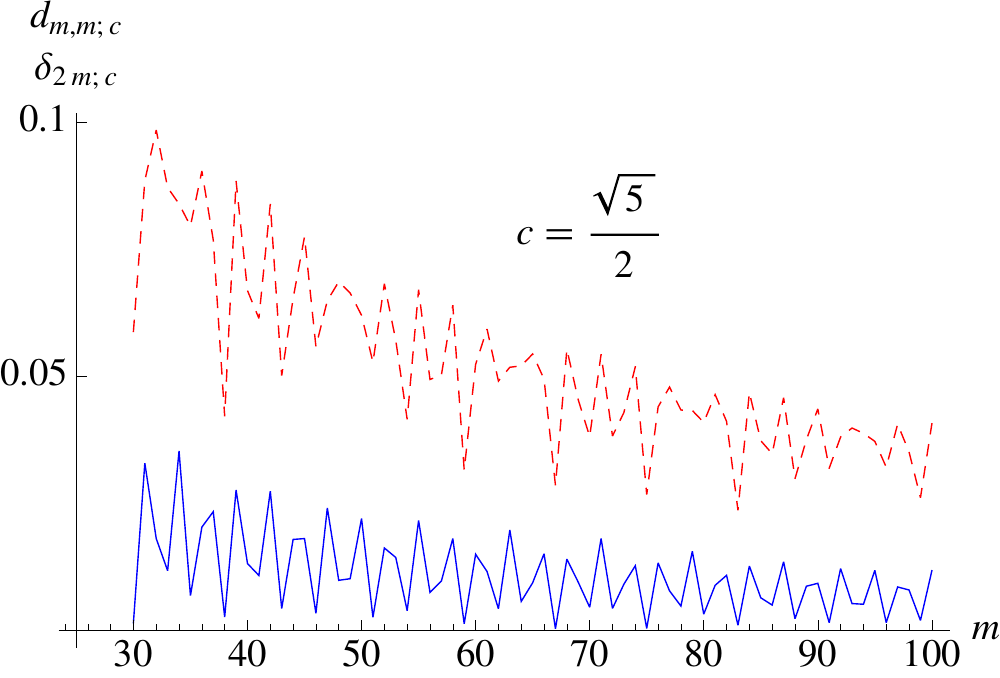}}

\

\fbox{\includegraphics{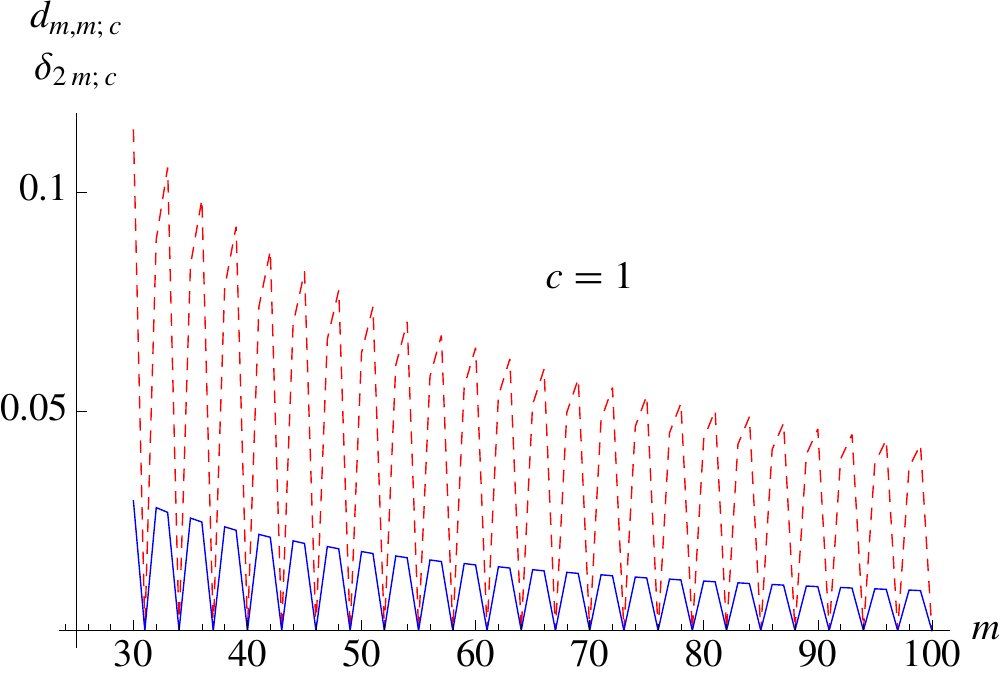}}
\caption{\small $d_{m,m;c}$ (dashed red line) and $\del_{2m;c}$ (solid blue line), drawn as functions of $m$ for $c=\sqrt{5}/2$ (top figure) and $c=1$ (bottom figure). The graphs in the bottom figure exhibit much more regular behaviour than those in the top one because
$\pi^{-1} \arccos(1/2)$ is rational, while $\pi^{-1} \arccos(\sqrt{5}/4)$ is not.\label{fig:pict8}}
\end{center}
\end{figure}

\appendix
\section{Ancillary materials}\label{sec:app}
The video file, {\tt frames3e.mp4} (approx. 4.4 MB), is available for download from the {\tt Ancillary files} section on the {\tt arXiv} page of this paper. This shows the dynamics of the eigenvalues of  $\Ap_{c;100}$ as $c$ diminishes from $2.05$ to $0$. The real eigenvalues are shown as blue dots, and the non-real ones as the red dots. Also, the asymptotic curves
$\pm \Lambda_c(|\Re(\lambda)|)$ are shown in black for $c<2$. 

The {\tt{\it Mathematica}${}^\text{\textregistered}$ 7} code used to produce the video file is available from the same section as {\tt makemovie.nb} and its listing as {\tt makemovie.pdf}.
\end{document}